\documentclass{article}
\usepackage{amsmath,amsthm,amsopn,amssymb}
\allowdisplaybreaks

\numberwithin{equation}{section}
\newtheorem{theorem}{Theorem}[section]
\newtheorem{lemma}{Lemma}[section]
\makeatother

\newcommand{\Rnn}{{\mathbb{R}}^{n\times n}}

\newcommand{\Rrn}{{\mathbb{R}}^{r\times n}}
\newcommand{\Rnr}{{\mathbb{R}}^{n\times r}}

\begin{document}

\title{A modified large-scale structure-preserving doubling algorithm for a large-scale Riccati equation from transport theory
\thanks{This research is supported by the National Natural Science Foundation under grant 11301013.} }

\author{Pei-Chang Guo \thanks{Corresponding author, e-mail: gpeichang@126.com} \\
School of Mathematical Sciences, Peking University, Beijing 100871, Beijing, China}

\date{}

\maketitle
\begin{abstract}
We consider the large scale nonsymmetric algebraic Riccati equation arising in
transport theory, where the $n\times n$ coefficient matrices $B, C$ are symmetric and low-ranked
and $A, E$ are rank one updates of nonsingular diagonal matrices. By introducing a
balancing strategy and setting appropriate initial matrices carefully, we can simplify the large-scale structure-preserving doubling algorithm (SDA\_ls) for this
special equation. We give modified large-scale structure-preserving doubling algorithm, which can
reduce the flop operations of original SDA\_ls by half. Numerical experiments illustrate the effectiveness of our method.

\vspace{2mm} \noindent \textbf{Keywords}: large-scale nonsymmetric algebraic
Riccati equation, large-scale structure-preserving doubling algorithm, balancing
strategy, appropriate initial matrices, transport theory.

\end{abstract}

\section{Introduction}
In nuclear physics, we need to solve the  nonsymmetric
algebraic Riccati equation (NARE):
\begin{equation}\label{lsriccati}
XCX-XE-AX+B=0
\end{equation}
and its dual equation
\begin{equation}\label{dualequation}
YBY-YA-EY+C=0 ,
\end{equation}
 where $A,B,C,E \in \Rnn$ are matrices given by
\begin{equation}\label{lscoematrix1}
A=\Delta-eq^T,B=ee^T,C=qq^T,E=D-qe^T,
\end{equation}
with
\begin{eqnarray}\label{lscoematrix2}
\left\{\begin{array}{ccll}
e&=&(1,1,\ldots,1)^T, \\
q &=& (q_1,q_2,\ldots,q_n)^T, \quad \quad \,\,q_i=\frac{c_i}{2\omega_i},\\
\Delta &=& \mbox{diag}(\delta_1,\delta_2,\ldots,\delta_n),
\quad \, \delta_i=\frac{1}{c\omega_i(1+\alpha)},\\
D &= &\mbox{diag}(d_1,d_2,\ldots,d_n), \quad
d_i=\frac{1}{c\omega_i(1-\alpha)},
\end{array}\right.
\end{eqnarray}
and $0<c\leq1$, $0\leq\alpha<1$,
$0<\omega_{n}<\cdots<\omega_{2}<\omega_{1}$,
$\sum_{i=1}^{n}c_{i}=1$, $c_{i}>0$, $i=1,2,\cdots,n$.

These matrices and vectors rely on the parameters $c$, $\alpha$,
$\omega_{i}$ and $c_{i}$. For the further physical meaning of these
parameters, please see \cite{jua95} and the references
therein. In this paper, NARE (\ref{lsriccati}) is always
referred to  the nonsymmetric algebraic Riccati equation
(\ref{lsriccati}) associated with the special structure given by
\eqref{lscoematrix1} and (\ref{lscoematrix2}). In applications from
transport theory, the minimal nonnegative solution of NARE \eqref{lsriccati} is of interest.
The numerical algorithms about this minimal nonnegative solution have been
studied by many authors in the past decade, and various direct and iterative
methods have been proposed. The minimal nonnegative solution is
associated with the matrix
\begin{eqnarray*}
H=\left[\begin{array}{lc}
E & -C \\
B & -A
\end{array}\right].
\end{eqnarray*}
The solution of NARE \eqref{lsriccati} can be expressed in
closed form in terms of the eigenvalues and eigenvectors of $H$; see
\cite{jua98} and \cite{xhg08}. Available iteration algorithms are
the Newton's method \cite{binew,guo01,guo00,lunew,gxx05}, the
fixed-point iteration method \cite{bzz08,bzz,guo10,jua98,lu05} and
the structure-preserving doubling algorithm \cite{shift,SDA,bisda,lssda}.

In \cite{SDA} the authors introduced a structure-preserving
transformation and developed a structure-preserving doubling
algorithm (SDA) approximating quadratically to the minimal
nonnegative solution of NARE(\ref{lsriccati}) and its dual
equation. For medium size problems without the special structure SDA is very effective. In \cite{lssda}, the authors propose a large-scale structure-preserving doubling algorithm (SDA\_ls) based on the sparse-plus-low-rank coefficient matrix structure, which has $O(n)$ computational complexity  and memory requirement per iteration and could be applied to large-scale NARE satisfying the sparse-like structure.

In this paper, we further utilize the special structure of coefficient matrices of NARE \eqref{lsriccati} and improve the large-scale structure-preserving doubling algorithm. There are four matrix sequences $\{E_k\},\{F_k\},\{H_k\}$ and $\{G_k\}$ when using SDA. After an appropriate balancing of the coefficient matrices , we design appropriate initial matrices of SDA\_ls, which can reduce the flop operations of SDA\_ls by half. Our modified SDA\_ls maintains the same quadratic convergence rate as original SDA\_ls and is more efficient on large-scale NARE \eqref{lsriccati}. The main contribution of this paper is that we show how to reduce the flop operations of SDA\_ls by half through introducing a balancing strategy and setting the appropriate initial matrices for SDA\_ls. We prove our result theoretically.

The rest of this paper is organized as follows. In section 2 some preliminaries
are presented. In section 3 we give the balancing strategy and the appropriate initial matrices for SDA\_ls, showing how to reduce the flop operations by half at
each iteration step. In section 4 we present some
numerical examples, which show that our modified SDA\_ls is much faster
than original SDA\_ls applied to NARE \eqref{lsriccati}. Throughout the paper, the Hadamard
product of A and B is defined by $A\circ B = (a_{ij}b_{ij})$, $I$ is used to denote the identity matrix of appropriate dimension, $\oplus$ is used to denote the direct sum of square matrices.

\section{Preliminaries}

\subsection{Structure-Preserving Doubling Algorithm}

The structure-preserving doubling algorithm \cite{SDA} is
quadratically convergent to computing the minimal positive solution of  NARE(\ref{lsriccati}). The algorithm can be described as
follows.

\vspace{2mm}

Choose $\gamma\geq$ max $\{e_{ii}, a_{ii}: i=1, \cdots, n \}$; let
\[
W=A+\gamma I-B(E+\gamma I)^{-1}C,   V=E+\gamma I-C(A+\gamma
I)^{-1}B,
\]
and
\begin{eqnarray}\label{lsinitial-matrix}
\left.\begin{array}{cc}
E_0=I-2\gamma V^{-1},  & F_0=I-2\gamma W^{-1}, \\
G_0=2\gamma (E+\gamma I)^{-1}CW^{-1}, &  H_0=2\gamma
W^{-1}B(E+\gamma I)^{-1}.
\end{array}\right.
\end{eqnarray}
where $A, B, C, E$ are coefficient matrices of
(\ref{lsriccati}) and $e_{ii}$ and $a_{ii}$ are the $i-th$ diagonal elements of the
matrices $E$ and $A$, respectively.

For $k\geq0$, calculate
\begin{eqnarray}\label{lsiteration-matrix}
\left.\begin{array}{ccc}
E_{k+1}&=&E_k(I-G_kH_k)^{-1}E_k, \\
F_{k+1}&=&F_k(I-H_kG_k)^{-1}F_k, \\
G_{k+1}&=&G_k+E_k(I-G_kH_k)^{-1}G_kF_k, \\
H_{k+1}&=&H_k+F_k(I-H_kG_k)^{-1}H_kE_k.
\end{array}\right.
\end{eqnarray}
Let
\begin{eqnarray*}
K=\left[\begin{array}{cc}
E & -C \\
-B & A
\end{array}\right]
\end{eqnarray*}
be a nonsingular $M$-matrix. If
$\lambda_1,\lambda_2,\ldots,\lambda_{2n}$ are the eigenvalues of $H$
ordered by nonincreasing real parts, then the eigenvalues of $K$ are
$\lambda_1,\ldots,\lambda_n,-\lambda_{n+1},\ldots,-\lambda_{2n}$\cite{guo01}.

From \cite{shift,Iann}, We have the following convergence result.
\begin{lemma}
If the matrix sequences $\{E_k\},\{F_k\},\{H_k\}$ and $\{G_k\}$ are
generated by (\ref{lsiteration-matrix}), and $K$ is nonsingular, then
 $\lim_{k \to \infty}E_k = \lim_{k \to \infty}F_k = 0$, $\lim_{k
\to \infty}G_k = Y$, and $\lim_{k \to \infty}H_k = X$ with quadratic
convergence rate, where $X$ and $Y$ are the minimal nonnegative
solutions of NARE (\ref{lsriccati}) and its dual equation
(\ref{dualequation}), respectively.
\end{lemma}

\subsection{Large-scale SDA}
To develop a large-scale SDA \cite{lssda}, The authors assume that $E$, $A$ are sparselike (with the products $A^{-1}u, A^{-T}u,  E^{-1}u, E^{-T}u$ computable in $O(n)$ complexity where $u$ is a vector) and $B,C\in \Rnn$ have the full low-ranked decompositions
\begin{equation}\label{lsassump}
    B=B_1B_2^T, C=C_1C_2^T,
\end{equation}
where $B_1 ,C_1\in \Rnr$ and $B_2 ,C_2\in \Rrn$ with $r\ll n$. The authors apply the Sherman-Morrison-Woodbury formula to avoid the invention of large matrices and use low ranked matrices to represent iterates. The basic large-scale SDA (SDA\_ls) is as follows.

For $k=1,2,\cdots,$ SDA\_ls is organized so that the iterates have the recursive forms
\begin{eqnarray}\label{lssdait}
 \nonumber H_k^{\tau} = Q_{1k}^{\tau} \Sigma_k^{\tau} Q_{2k}^{\tau T}, &  G_k^{\tau} = P_{1k}^{\tau} \Gamma_k^{\tau} P_{2k}^{\tau T},\\
  E_k^{\tau} = E_{k-1}^{\tau 2}+E_{1k}^{\tau}E_{2k}^{\tau T},&  F_k^{\tau} = F_{k-1}^{\tau 2}+F_{1k}^{\tau}F_{2k}^{\tau T},
\end{eqnarray}
where $ Q_{ik}^{\tau}$ and $ P_{ik}^{\tau}$ are column orthogonal with widths being $m_k$ and $l_k$ $(i=1,2)$ respectively. Note that $ E_k^{\tau}$ and $ F_k^{\tau}$ are not formed explicitly.

For the initial matrices \eqref{lsinitial-matrix}, we have
\begin{equation*}
    E_0=I-2\gamma V^{-1}, F_0=I-2\gamma W^{-1}, H_0=Q_{10}\Sigma_0 Q_{20}^T, G_0=P_{10}\Gamma_0P_{20}^T,
\end{equation*}
where
\begin{eqnarray}\label{lsinitial}
  Q_{10} \equiv  2 \gamma W^{-1}B_1,& Q_{20} \equiv (E+\gamma I)^{-T}B_2,& \Sigma_0\equiv I,\\
  P_{10} \equiv  2 \gamma (E+\gamma I)^{-1}C_1,& P_{20} \equiv W^{-T}C_2,& \Gamma_0\equiv I.
\end{eqnarray}
Note that efficient linear solvers with $O(n)$ complexity for large scale $A, E$ and $A+\gamma I, E+\gamma I$ are available. Utilizing Shermann-Morrison-Woodbury (SMW) formula $W^{-1}u$ and $V^{-1}u$ can be computed with $O(n)$ complexity too.
Then we compute the economic QR decompositions
\begin{equation}\label{lsqrini}
    Q_{10}=\hat{Q}_{10} R_{1q}, Q_{20}=\hat{Q}_{20} R_{2q}, P_{10}=\hat{P}_{10} R_{1p}, P_{20}=\hat{P}_{20} R_{2p},
\end{equation}
where $\hat{Q}_{10}, \hat{Q}_{20}, \hat{P}_{10}, \hat{P}_{20}$ are column orthogonal and $R_{1q}, R_{2q}, R_{1p}, R_{2p}$  are upper triangular matrices with positive diagonal entries. Then we compute the SVDs
\begin{eqnarray}\label{lssvdini}
\nonumber  R_{1q}\Sigma_0 R_{2q}^T&=&[U_{10}^{\tau}, U_{10}^{\epsilon}](\Sigma_0^{\tau}\oplus \Sigma_0^{\epsilon})[U_{20}^{\tau}, U_{20}^{\epsilon}]^T ,\|\Sigma_0^{\epsilon}\|<\epsilon_0,\\
   R_{1p}\Gamma_0 R_{2p}^T&=&[V_{10}^{\tau}, V_{10}^{\epsilon}](\Gamma_0^{\tau}\oplus \Gamma_0^{\epsilon})[V_{20}^{\tau}, V_{20}^{\epsilon}]^T ,\|\Gamma_0^{\epsilon}\|<\epsilon_0,
\end{eqnarray}
where
$[U_{i0}^{\tau}, U_{i0}^{\epsilon}]$ and $[V_{20}^{\tau}, V_{20}^{\epsilon}]$ $(i=1,2)$ are orthogonal, $\Sigma_0^{\tau}\oplus \Sigma_0^{\epsilon}$ and $\Gamma_0^{\tau}\oplus \Gamma_0^{\epsilon}
$ are nonnegative diagonal.
By setting
\begin{equation}\label{lsinit}
    Q_{10}^{\tau}=\hat{Q}_{10}U_{10}^{\tau},  Q_{20}^{\tau}=\hat{Q}_{20}U_{20}^{\tau},  P_{10}^{\tau}=\hat{P}_{10}V_{10}^{\tau},  P_{20}^{\tau}=\hat{P}_{20}V_{20}^{\tau},
\end{equation}
we get the truncated initial matrices for large-scale SDA (SDA\_ls)
\begin{equation}\label{lsinit2}
    E_{0}^{\tau}=E_0, F_0^{\tau}=F_0, H_0^{\tau} = Q_{10}^{\tau} \Sigma_0^{\tau} Q_{20}^{\tau T},   G_0^{\tau} = P_{10}^{\tau} \Gamma_0^{\tau} P_{20}^{\tau T}.
\end{equation}

Then we compute
\begin{eqnarray}\label{lssigma}
\nonumber  \breve{\Sigma}_k^{\tau}&=& \Sigma_k^{\tau}+\Sigma_k^{\tau} Q_{2k}^{\tau T} P_{1k}^{\tau}\Gamma_k^{\tau}(I- P_{2k}^{\tau T}H_k^{\tau}P_{1k}^{\tau}\Gamma_k^{\tau})^{-1} P_{2k}^{\tau T} Q_{1k}^{\tau}\Sigma_k^{\tau},\\
  \breve{\Gamma}_k^{\tau}&=& \Gamma_k^{\tau}+(I- \Gamma_k^{\tau} P_{2k}^{\tau T} H_k^{\tau}P_{1k}^{\tau})^{-1}\Gamma_k^{\tau}P_{2k}^{\tau T}H_k^{\tau}P_{1k}^{\tau}\Gamma_k^{\tau},
\end{eqnarray}
and \begin{eqnarray}
    \nonumber  F_{1,k+1}^{\tau}=F_k^{\tau} Q_{1k}^{\tau}(I- \Sigma_k^{\tau} Q_{2k}^{\tau T} G_k^{\tau}Q_{1k}^{\tau})^{-1}\Sigma_k^{\tau} Q_{2k}^{\tau T} P_{1k}^{\tau}\Gamma_k^{\tau}, & F_{2,k+1}^{\tau}=F_k^{\tau T} P_{2k}^{\tau},\\
      E_{1,k+1}^{\tau}=E_k^{\tau} P_{1k}^{\tau} (I- \Gamma_k^{\tau} P_{2k}^{\tau T} H_k^{\tau}P_{1k}^{\tau})^{-1}\Gamma_k^{\tau}P_{2k}^{\tau T}Q_{1k}^{\tau}\Sigma_k^{\tau}, &  E_{2,k+1}^{\tau}=E_k^{\tau T} Q_{2k}^{\tau}.
    \end{eqnarray}
    From the economic QR decompositions
     \begin{equation}\label{lsQkplus1}
     [Q_{1k}^{\tau},F_k^{\tau} Q_{1k}^{\tau}] = [Q_{1k}^{\tau},\hat{Q}_{1k}]\left[\begin{array}{lc}
     I & S_{1q}\\
     0 & R_{1q}
   \end{array}\right], [Q_{2k}^{\tau},E_k^{\tau} Q_{2k}^{\tau}] = [Q_{2k}^{\tau},\hat{Q}_{2k}]\left[\begin{array}{lc}
   I & S_{2q}\\
    0 & R_{2q}
   \end{array}\right],
\end{equation}
\begin{equation}\label{lsQkplus2}
[P_{1k}^{\tau},E_k^{\tau} P_{1k}^{\tau}] = [P_{1k}^{\tau},\hat{P}_{1k}]\left[\begin{array}{lc}
I & S_{1p}\\
0 & R_{1p}
\end{array}\right],  [P_{2k}^{\tau},F_k^{\tau} P_{2k}^{\tau}] = [P_{2k}^{\tau},\hat{P}_{2k}]\left[\begin{array}{lc}
I & S_{2p}\\
0 & R_{2p}
\end{array}\right],
\end{equation}
we compute
\begin{eqnarray}\label{lssigmhat}
 \nonumber  \hat{\Sigma}_{k+1} &=& \left[\begin{array}{lc}
I & S_{1q}\\
0 & R_{1q}
\end{array}\right]\left[\begin{array}{lc}
\Sigma_k^{\tau} & 0\\
0 & \breve{\Sigma}_k^{\tau}
\end{array}\right] \left[\begin{array}{lc}
I & S_{2q}\\
0 & R_{2q}
\end{array}\right]^T \\
\hat{\Gamma}_{k+1}&=&\left[\begin{array}{lc}
I & S_{1p}\\
0 & R_{1p}
\end{array}\right]\left[\begin{array}{lc}
\Gamma_k^{\tau} & 0\\
0 & \breve{\Gamma}_k^{\tau}
\end{array}\right] \left[\begin{array}{lc}
I & S_{2p}\\
0 & R_{2p}
\end{array}\right]^T,
\end{eqnarray}
then we compute the SVDs
\begin{equation*}
\nonumber   \hat{\Sigma}_{k+1}=\left[\begin{array}{lc}
 U_{1,k+1}^{\tau} & U_{1,k+1}^{\epsilon}
\end{array}\right](\Sigma_{k+1}^{\tau}\oplus \Sigma_{k+1}^{\epsilon})\left[\begin{array}{lc}
U_{2,k+1}^{\tau} & U_{2,k+1}^{\epsilon}\end{array}\right]^T
\end{equation*}
\begin{equation}\label{lssvd}
 \hat{\Gamma}_{k+1}=\left[\begin{array}{lc}
 V_{1,k+1}^{\tau} & V_{1,k+1}^{\epsilon}
\end{array}\right](\Gamma_{k+1}^{\tau}\oplus \Gamma_{k+1}^{\epsilon})\left[\begin{array}{lc}
V_{2,k+1}^{\tau}
& V_{2,k+1}^{\epsilon}\end{array}\right]^T
\end{equation}
with $\|\Sigma_{k+1}^{\epsilon}\|<\epsilon_{k+1}$ and $\|\Gamma_{k+1}^{\epsilon}\|<\epsilon_{k+1}$.
The truncated matrices are
\begin{eqnarray}\label{lsqpupdate}
\nonumber Q_{1,k+1}^{\tau} = [Q_{1k}^{\tau},\hat{Q}_{1k}]U_{1,k+1}^{\tau} ,& Q_{2,k+1}^{\tau} = [Q_{2k}^{\tau},\hat{Q}_{2k}]U_{2,k+1}^{\tau} ,\\
  P_{1,k+1}^{\tau} = [P_{1k}^{\tau},\hat{P}_{1k}]V_{1,k+1}^{\tau} ,& P_{2,k+1}^{\tau} = [P_{2k}^{\tau},\hat{P}_{2k}]V_{2,k+1}^{\tau}.
\end{eqnarray}
We note that if QR decompositions and SVDs are not introduced to truncate $Q_{ik}, P_{ik} (i=1,2)$ in SDA\_ls, then SDA\_ls is mathematically equivalent to SDA \eqref{lsiteration-matrix} because \eqref{lssdait} are derived from \eqref{lsiteration-matrix} by applying SMW formula:
\begin{eqnarray*}
  (I-H_k^{\tau}G_k^{\tau})^{-1} &=&I+Q_{1k}^{\tau}(I- \Sigma_k^{\tau} Q_{2k}^{\tau T} G_k^{\tau}Q_{1k}^{\tau})^{-1}\Sigma_k^{\tau} Q_{2k}^{\tau T} G_k^{\tau},\\
 (I-G_k^{\tau}H_k^{\tau})^{-1} &=&I+P_{1k}^{\tau} (I- \Gamma_k^{\tau} P_{2k}^{\tau T} H_k^{\tau}P_{1k}^{\tau})^{-1}\Gamma_k^{\tau}P_{2k}^{\tau T}H_k^{\tau}.
\end{eqnarray*}

\section{Balancing Strategy and Modified SDA\_ls}

In this section we introduce a balancing strategy, with
which the matrix $\widetilde{X}^T$  is the minimal positive solution
of the dual equation if $\widetilde{X}$ is the minimal
positive solution of the algebraic Riccati equation.

Because the entries of the vector $q$ are positive, we may define
\begin{eqnarray*}
\Phi := \mbox{diag}(\sqrt{q_1},\cdots,\sqrt{q_n}),    \phi :=
{(\sqrt{q_1},\cdots,\sqrt{q_n})}^T.
\end{eqnarray*}
Let
\begin{eqnarray*}
\tilde{X} &=& \Phi X\Phi,\\
\tilde{A} &=&\Phi A{\Phi}^{-1}=\Delta- \phi{\phi}^T,\\
\tilde{B} &=&\Phi B\Phi=\phi{\phi}^T,\\
\tilde{C} &= &{\Phi}^{-1} C{\Phi}^{-1}=\phi{\phi}^T=\tilde{B},\\
\tilde{E} &=& {\Phi}^{-1} E\Phi=D- \phi{\phi}^T.
\end{eqnarray*}
Then the algebraic Riccati equation of (\ref{lsriccati}) can be
equivalently expressed as
\begin{equation}\label{lsbalriccati}
\tilde{X}\tilde{C}\tilde{X}-\tilde{X}\tilde{E}-\tilde{A}\tilde{X}+\tilde{B}=0.
\end{equation}
Obviously, $X$ is a solution of (\ref{lsriccati}) if and only
if $\tilde{X} = \Phi X\Phi$ is a solution of
(\ref{lsbalriccati}).

Let
\begin{eqnarray*}
\tilde{K}&=&\left[\begin{array}{cc}
\tilde{E} & -\tilde{C} \\
-\tilde{B} & \tilde{A}
\end{array}\right]\\
 &= &\left[\begin{array}{cc}
{\Phi}^{-1}& 0 \\
0     & \Phi
\end{array}\right]K\left[\begin{array}{cc}
\Phi & 0 \\
0 & {\Phi}^{-1}
\end{array}\right].
\end{eqnarray*}

Then  $\tilde{K}$ is similar to $K$.


Choose $\tilde{\gamma}\geq$ max$\{ \tilde{e}_{ii}, \tilde{a}_{ii}:
i=1, \cdots, n \}$; let
\[
\tilde{W}=\tilde{A}+\tilde{\gamma}
I-\tilde{B}(\tilde{E}+\tilde{\gamma} I)^{-1}\tilde{C},
\tilde{V}=\tilde{E}+\tilde{\gamma}
I-\tilde{C}(\tilde{A}+\tilde{\gamma} I)^{-1}\tilde{B}
\]
and
\begin{eqnarray}\label{lsinitialmatrix}
\left.\begin{array}{cc}
\tilde{E}_0=I-2\tilde{\gamma} \tilde{V}^{-1}, &  \tilde{F}_0=I-2\tilde{\gamma} \tilde{W}^{-1}, \\
\tilde{G}_0=2\tilde{\gamma} (\tilde{E}+\tilde{\gamma}
I)^{-1}\tilde{C}\tilde{W}^{-1},& \tilde{H}_0=2\tilde{\gamma}
\tilde{W}^{-1}\tilde{B}(\tilde{E}+\tilde{\gamma} I)^{-1},
\end{array}\right.
\end{eqnarray}
where $\widetilde{A},\widetilde{B},\widetilde{C}$ and
$\widetilde{E}$ are coefficient matrices of
(\ref{lsbalriccati}) and $\tilde{e}_{ii}$ and $\tilde{a}_{ii}$
are the $i-th$ diagonal elements of the matrices $\tilde{E}$ and
$\tilde{A}$, respectively..

For $k\geq0$, calculate
\begin{eqnarray}\label{lsiterationmatrix}
\left.\begin{array}{ccl}
\tilde{E}_{k+1}&=&\tilde{E}_k(I-\tilde{G}_k\tilde{H}_k)^{-1}\tilde{E}_k, \\
\tilde{F}_{k+1}&=&\tilde{F}_k(I-\tilde{H}_k\tilde{G}_k)^{-1}\tilde{F}_k, \\
\tilde{G}_{k+1}&=&\tilde{G}_k+\tilde{E}_k(I-\tilde{G}_k\tilde{H}_k)^{-1}\tilde{G}_k
\tilde{F}_k, \\
\tilde{H}_{k+1}&=&\tilde{H}_k+\tilde{F}_k(I-\tilde{H}_k\tilde{G}_k)^{-1}\tilde{H}_k\tilde{E}_k.
\end{array}\right.
\end{eqnarray}

For the initial matrices of SDA\_ls for NARE \eqref{lsbalriccati},  if we set
\begin{equation*}
    E_0=I-2\tilde{\gamma} \tilde{V}^{-1},   F_0=I-2\tilde{\gamma} \tilde{W}^{-1},  H_0=Q_{10}\Sigma_0 Q_{20}^T, G_0=P_{10}\Gamma_0P_{20}^T,
\end{equation*}
where
\begin{eqnarray*}
  Q_{10} \equiv  \sqrt{2 \tilde{\gamma}} \tilde{W}^{-1}\phi,& Q_{20} \equiv \sqrt{2 \tilde{\gamma}} (\tilde{E}+\tilde{\gamma} I)^{-T}\phi,& \Sigma_0\equiv I,\\
  P_{10} \equiv \sqrt{2 \tilde{\gamma}} (\tilde{E}+\tilde{\gamma} I)^{-1}\phi,& P_{20} \equiv \sqrt{2  \tilde{\gamma}}\tilde{W}^{-T}\phi,& \Gamma_0\equiv I,
\end{eqnarray*}
then the flop operations of large scale structure preserving doubling algorithm can be reduced by half. We show this in Theorem \ref{lsthm1}.

\begin{theorem} \label{lsthm1}
For the initial matrices of SDA\_ls applied to NARE \eqref{lsbalriccati},  if
\begin{equation*}
    E_0=I-2\tilde{\gamma} \tilde{V}^{-1},   F_0=I-2\tilde{\gamma} \tilde{W}^{-1},  H_0=Q_{10}\Sigma_0 Q_{20}^T, G_0=P_{10}\Gamma_0P_{20}^T,
\end{equation*}
where
\begin{eqnarray}\label{lsmodinitial}
  Q_{10} \equiv  \sqrt{2 \tilde{\gamma}} \tilde{W}^{-1}\phi,& Q_{20} \equiv \sqrt{2 \tilde{\gamma}} (\tilde{E}+\tilde{\gamma} I)^{-T}\phi,& \Sigma_0\equiv I,\\
  P_{10} \equiv \sqrt{2 \tilde{\gamma}} (\tilde{E}+\tilde{\gamma} I)^{-1}\phi,& P_{20} \equiv \sqrt{2  \tilde{\gamma}} \tilde{W}^{-T}\phi,& \Gamma_0\equiv I,
\end{eqnarray}
then for SDA\_ls
\begin{eqnarray}
  Q_{1k}^{\tau}= P_{2k}^{\tau}, & Q_{2k}^{\tau}=P_{1k}^{\tau},& \Sigma_k^{\tau}=\Gamma_k^{\tau},  \\
  E_k^{\tau} = E_k^{\tau T} & F_k^{\tau} = F_k^{\tau T}, & H_k^{\tau} = G_k^{\tau T},
\end{eqnarray}
\begin{equation}
    E_k^{\tau} P_{1k}^{\tau}=E_k^{\tau T} Q_{2k}^{\tau},  F_k^{\tau} Q_{1k}^{\tau}=F_k^{\tau T} P_{2k}^{\tau}
\end{equation}
hold true for $k=1,2,\cdots$.

\end{theorem}
\begin{proof}
We prove the result by mathematical induction.

After being balanced, Since $\tilde{B}=\tilde{C}$, $\tilde{A}^T=\tilde{A}$ and
$\tilde{E}^T=\tilde{E}$, we have
\[
\tilde{W}^T=\tilde{W},  \tilde{V}^T=\tilde{V},
\]
and, thus,
\begin{equation*}
  E_0^T=E_0, F_0^T=F_0, H_{0}=G_{0}^T.
\end{equation*}

From \begin{eqnarray*}
 Q_{10} \equiv  \sqrt{2 \tilde{\gamma}} \tilde{W}^{-1}\phi,& Q_{20} \equiv \sqrt{2 \tilde{\gamma}} (\tilde{E}+\tilde{\gamma} I)^{-T}\phi,& \Sigma_0\equiv I,\\
  P_{10} \equiv \sqrt{2 \tilde{\gamma}} (\tilde{E}+\tilde{\gamma} I)^{-1}\phi,& P_{20} \equiv \sqrt{2  \tilde{\gamma}} W^{-T}\phi,& \Gamma_0\equiv I,
\end{eqnarray*}
we get
\begin{equation*}
    Q_{10}= P_{20},  Q_{20}=P_{10}, \Sigma_0=\Gamma_0^T.
\end{equation*}
Then from the initial process of SDA\_ls, i.e. \eqref{lsqrini},\eqref{lssvdini}, \eqref{lsinit}, \eqref{lsinit2}, we can verify that
\begin{eqnarray*}
  Q_{10}^{\tau}= P_{20}^{\tau}, & Q_{20}^{\tau}=P_{10}^{\tau},& \Sigma_0^{\tau}=\Gamma_0^{\tau T},  \\
  E_0^{\tau} = E_0^{\tau T} & F_0^{\tau} = F_0^{\tau T}, & H_0^{\tau} = G_0^{\tau T}, \\
\end{eqnarray*}
and
\begin{equation*}
    E_0^{\tau} P_{10}^{\tau}=E_0^{\tau T} Q_{20}^{\tau}, F_0^{\tau} Q_{10}^{\tau}=F_0^{\tau T} P_{20}^{\tau}.
\end{equation*}

Hence the result is true for $k=0$. Assume that the result holds true for $k=l$, and we consider the case of $k=l+1$.
Note that at each local step of SDA\_ls, the local truncation is done only on matrices $Q_{ik}$ and $P_{ik}$ (i=1,2), so
the following iteration processes of SDA\_ls
\begin{equation*}
    E_{l+1}^{\tau} = E_{l}^{\tau 2}+E_{1,l+1}^{\tau}E_{2,l+1}^{\tau T},  F_{l+1}^{\tau} = F_{l}^{\tau 2}+F_{1,l+1}^{\tau}F_{2,l+1}^{\tau T},
\end{equation*}
are mathematically equivalent to applying SMW formula to
\begin{equation}\label{lssdaefit}
    E_{l+1}^{\tau}=E_l^{\tau}(I-G_l^{\tau}H_l^{\tau})^{-1}E_l^{\tau},
F_{l+1}^{\tau}=F_l^{\tau}(I-H_lG_l^{\tau})^{-1}F_l).
\end{equation}

From the assumption of induction we know
\begin{equation*}
   E_{l}^{\tau}=E_{l}^{\tau T}, \quad F_{l}^{\tau}=F_{l}^{\tau T} ,\quad  H_{l}^{\tau}=G_{l}^{\tau T} ,
\end{equation*}
so from \eqref{lssdaefit} we get
\begin{equation*}
   E_{l+1}^{\tau}=E_{l+1}^{\tau T},  F_{l+1}^{\tau}=F_{l+1}^{\tau T} .
\end{equation*}
By direct calculations we can verify that
\begin{equation*}
    \breve{\Sigma}_{l}^{\tau}=\breve{\Gamma}_l^{\tau T}
\end{equation*}
holds in the iteration step \eqref{lssigma} and we have
\begin{equation*}
    [Q_{1l}^{\tau},F_l^{\tau} Q_{1l}^{\tau}] =[P_{2l}^{\tau},F_l^{\tau} P_{2l}^{\tau}], [Q_{2l}^{\tau},E_l^{\tau} Q_{2l}^{\tau}]= [P_{1l}^{\tau},E_l^{\tau} P_{1l}^{\tau}]
\end{equation*}
in iteration step \eqref{lsQkplus1} because of assumption of induction, so we get
\begin{equation*}
     \hat{\Sigma}_{l+1}=   \hat{\Gamma}_{l+1}^T
\end{equation*}
in the iteration step \eqref{lssigmhat}. Then from \eqref{lssvd} we get
\begin{equation*}
    \Sigma_{l+1}^{\tau}=\Gamma_{l+1}^{\tau T}
\end{equation*}
and from \eqref{lsqpupdate}
we get
\begin{equation*}
    Q_{1,l+1}^{\tau} =P_{2,l+1}^{\tau} , Q_{2,l+1}^{\tau}= P_{1,l+1}^{\tau},
\end{equation*}
and thus
\begin{equation*}
    H_{l+1}^{\tau}=G_{l+1}^{\tau T}
\end{equation*}

Combining with
\begin{equation*}
   E_{l+1}^{\tau}=E_{l+1}^{\tau T},  F_{l+1}^{\tau}=F_{l+1}^{\tau T} ,
\end{equation*}
we see that
\begin{equation*}
    E_{l+1}^{\tau} P_{1,l+1}^{\tau}=E_{l+1}^{\tau T} Q_{2,l+1}^{\tau}, F_{l+1}^{\tau} Q_{1,l+1}^{\tau}=F_{l+1}^{\tau T} P_{2,l+1}^{\tau}.
\end{equation*}
So we know that the conclusion holds true for $k=l+1$. By induction we have proved the result.
\end{proof}

Based on the above discussion, we present the modified SDA\_ls algorithm
for NARE (\ref{lsriccati}) as follows.

\vspace{3mm}

\textbf{Modified SDA\_ls for NARE (\ref{lsriccati})}

\begin{itemize}
\item[1.] Set $\phi :={(\sqrt{q_1},\cdots,\sqrt{q_n})}^T$ and let
\[
A=\Delta- \phi{\phi}^T, B=\phi{\phi}^T, C= \phi{\phi}^T, E= D-
\phi{\phi}^T.
\]

\item[2.]$k: =0$, compute
$W^{-1}, V^{-1}$
and $E_0, F_0$ implicitly and set
\begin{equation*}
  Q_{10} \equiv  \sqrt{2 \gamma} W^{-1}\phi, Q_{20} \equiv \sqrt{2 \gamma} (E+\gamma)^{-T}\phi, \Sigma_0\equiv I.
\end{equation*}
Compute economic QR decompositions
\begin{equation*}
    Q_{10}=\hat{Q}_{10} R_{1q}, Q_{20}=\hat{Q}_{20} R_{2q}
\end{equation*}
and compute the SVDs
\begin{equation*}
  R_{1q}\Sigma_0 R_{2q}^T=[U_{10}^{\tau}, U_{10}^{\epsilon}](\Sigma_0^{\tau}\oplus \Sigma_0^{\epsilon})[U_{20}^{\tau}, U_{20}^{\epsilon}]^T ,\|\Sigma_0^{\epsilon}\|<\epsilon_0.
\end{equation*}
Set
\begin{equation*}
    Q_{10}^{\tau}=\hat{Q}_{10}U_{10}^{\tau},  Q_{20}^{\tau}=\hat{Q}_{20}U_{20}^{\tau},
\end{equation*}
\begin{equation*}
    E_{0}^{\tau}=E_0, F_0^{\tau}=F_0, H_0^{\tau} = Q_{10}^{\tau} \Sigma_0^{\tau} Q_{20}^{\tau T}.
\end{equation*}

\item[3.] For $k\geq 0$, calculate
\begin{equation*}
 \breve{\Sigma}_k^{\tau}=\Sigma_k^{\tau}+\Sigma_k^{\tau} Q_{2k}^{\tau T} Q_{2k}^{\tau}\Gamma_k^{\tau}(I- Q_{1k}^{\tau T}H_k^{\tau}Q_{2k}^{\tau}\Gamma_k^{\tau})^{-1} Q_{1k}^{\tau T} Q_{1k}^{\tau}\Sigma_k^{\tau},
\end{equation*}
and \begin{eqnarray*}
   F_{2,k+1}^{\tau}=F_k^{\tau} Q_{1k}^{\tau}, &  F_{1,k+1}^{\tau}=F_{2,k+1}^{\tau}(I- \Gamma_k^{\tau} Q_{2k}^{\tau T} H_k^{\tau T}Q_{1k}^{\tau})^{-1}\Sigma_k^{\tau} Q_{2k}^{\tau T} P_{1k}^{\tau}\Gamma_k^{\tau}, \\
    E_{2,k+1}^{\tau}=E_k^{\tau} Q_{2k}^{\tau}, &
      E_{1,k+1}^{\tau}= E_{2,k+1}^{\tau} (I- \Gamma_k^{\tau} Q_{1k}^{\tau T} H_k^{\tau}P_{1k}^{\tau})^{-1}\Gamma_k^{\tau}Q_{1k}^{\tau T}Q_{1k}^{\tau}\Sigma_k^{\tau}.
    \end{eqnarray*}
    Compute economic QR decompositions
     \begin{equation*}
     [Q_{1k}^{\tau},F_k^{\tau} Q_{1k}^{\tau}] = [Q_{1k}^{\tau},\hat{Q}_{1k}]\left[\begin{array}{lc}
     I & S_{1q}\\
     0 & R_{1q}
   \end{array}\right], [Q_{2k}^{\tau},E_k^{\tau} Q_{2k}^{\tau}] = [Q_{2k}^{\tau},\hat{Q}_{2k}]\left[\begin{array}{lc}
   I & S_{2q}\\
    0 & R_{2q}
   \end{array}\right],
\end{equation*}
and
\begin{equation*}
 \hat{\Sigma}_{k+1} = \left[\begin{array}{lc}
I & S_{1q}\\
0 & R_{1q}
\end{array}\right]\left[\begin{array}{lc}
\Sigma_k^{\tau} & 0\\
0 & \breve{\Sigma}_k^{\tau}
\end{array}\right] \left[\begin{array}{lc}

I & S_{2q}\\
0 & R_{2q}
\end{array}\right]^T.
\end{equation*}
Then compute the SVDs
\begin{equation*}
  \hat{\Sigma}_{k+1}=\left[\begin{array}{lc}
 U_{1,k+1}^{\tau} & U_{1,k+1}^{\epsilon}
\end{array}\right](\Sigma_{k+1}^{\tau}\oplus \Sigma_{k+1}^{\epsilon})\left[\begin{array}{lc}
U_{2,k+1}^{\tau}, U_{2,k+1}^{\epsilon}\end{array}\right]^T
\end{equation*}
with $\|\Sigma_{k+1}^{\epsilon}\|<\epsilon_{k+1}$
and the truncated matrices
\begin{equation*}
 Q_{1,k+1}^{\tau} = [Q_{1k}^{\tau},\hat{Q}_{1k}]U_{1,k+1}^{\tau} , Q_{2,k+1}^{\tau} = [Q_{2k}^{\tau},\hat{Q}_{2k}]U_{2,k+1}^{\tau} ,
\end{equation*}
and form
\begin{equation*}
     E_k^{\tau} = E_{k-1}^{\tau 2}+E_{1k}^{\tau}E_{2k}^{\tau T},  F_k^{\tau} = F_{k-1}^{\tau 2}+F_{1k}^{\tau}F_{2k}^{\tau T}
\end{equation*}
implicitly.

\item[4.] If the stoping criterion is satisfied, then return
\begin{equation*}
    X=(\Phi ^{-1} Q_{1k}^{\tau}) \Sigma_k^{\tau} (Q_{2k}^{\tau T}
\Phi ^{-1}),
\end{equation*}
 where
${\phi}^{-1}={(\frac{1}{\sqrt{q_1}},\cdots,\frac{1}{\sqrt{q_n}})}^T$;
else go to step 3.
\end{itemize}
We compare the operation counts for the kth iteration in SDA\_ls and modified SDA\_ls. From Table 1 and Table 2, we see clearly the flops of modified SDA\_ls are roughly half of that of SDA\_ls.

\begin{table}\label{lstab1}
\begin{center}
\caption{Flop operations for the kth iteration in SDA\_ls}
\begin{tabular}{|l|c|}
  \hline
  Computation & Flops \\
  \hline
  $\breve{\Sigma}_k^{\tau},  \breve{\Gamma}_k^{\tau}$& $4m_km_kn$ \\
  $E_k^{\tau} P_{1k}^{\tau}, E_k^{\tau T} Q_{2k}^{\tau}$& $2[2^k(c_{\gamma}+4\sum_{j=1}^{k} 2^{-i}m_j] m_kn$ \\
  $F_k^{\tau} Q_{1k}^{\tau}, F_k^{\tau T} P_{2k}^{\tau}$& $2[2^k(c_{\gamma}+4\sum_{j=1}^{k} 2^{-i}m_j] m_kn$ \\
 $ F_{1,k+1}^{\tau}, E_{1,k+1}^{\tau}$ & $4m_km_kn$\\
 Orthogonalize $F_k^{\tau} Q_{1k}^{\tau}, F_k^{\tau T} P_{2k}^{\tau},E_k^{\tau} P_{1k}^{\tau}, E_k^{\tau T} Q_{2k}^{\tau}$ & $2[6(m_k^2+m_k^2)+2m_k]n$ \\
   $Q_{i,k+1}^{\tau},  P_{i,k+1}^{\tau}  (i=1,2)$& $16(m_km_{k+1})n$ \\
  \hline

\end{tabular}
\end{center}
\end{table}
\begin{table}\label{lstab2}
\begin{center}
\caption{Flop operations for the kth iteration in modified SDA\_ls}
\begin{tabular}{|l|c|}
  \hline
  Computation & Flops \\
  \hline
  $\breve{\Sigma}_k^{\tau}$ & $2m_km_kn$ \\
  $ E_k^{\tau} Q_{2k}^{\tau}$& $[2^k(c_{\gamma}+4\sum_{j=1}^{k} 2^{-i}m_j] m_kn$ \\
  $F_k^{\tau} Q_{1k}^{\tau}$& $[2^k(c_{\gamma}+4\sum_{j=1}^{k} 2^{-i}m_j] m_kn$ \\
 $ F_{1,k+1}^{\tau}, E_{1,k+1}^{\tau}$ & $4m_km_kn$\\
 Orthogonalize $F_k^{\tau} Q_{1k}^{\tau},  E_k^{\tau} Q_{2k}^{\tau}$ & $[6(m_k^2+m_k^2)+2m_k]n$ \\
   $Q_{i,k+1}^{\tau}  (i=1,2)$ & $8(m_km_{k+1})n$ \\
  \hline

\end{tabular}
\end{center}
\end{table}

\section{Numerical Experiments}

\section{Conclusions}
We have presented a balancing strategy and specially chosen initial matrices for the large-scale structure-preserving doubling algorithm applied to NARE \eqref{lsriccati}. We prove the flop operations of SDA\_ls at each iteration step can be reduced by half for this special equation. Computationally it is an interesting scheme that helps reducing flop operations and
computing the solution faster. Numerical experiments show that modified SDA\_ls for NARE \eqref{lsriccati} is very effective and outperforms SDA\_ls.

\end{document}